\documentclass[11pt]{article}
\RequirePackage{amsthm,amsmath, amssymb}
 \usepackage{epsfig}
 \usepackage{graphicx}
 \usepackage{pict2e}
 \usepackage{lmodern,microtype,hyperref}
  \newcommand{\Ex}{{\mathbb E}}
 \renewcommand{\Pr}{{\mathbb{P}}}
\newcommand{\bx}{\mathbf{x}}

\newcommand{\by}{\mathbf{y}}

\newcommand{\bs}{\mathbf{s}}

\newcommand{\Reals}{\mathbb{R}}
 \newcommand{\sfrac}[2]{{\textstyle\frac{#1}{#2}}}
\newcommand{\XX}{\mathcal{X}}

\newcommand{\eps}{\varepsilon}

\newtheorem{Lemma}{Lemma}

 \newtheorem{Theorem}[Lemma]{Theorem}
 \newtheorem{Conjecture}[Lemma]{Conjecture}

 \title{To stay discovered: On tournament mean score sequences and the Bradley--Terry model}
 \author{David J. Aldous\thanks{Research supported by NSF Grant DMS-1504802.}     \qquad  Brett Kolesnik\thanks{Research supported by
 an NSERC Postdoctoral Fellowship.} \\
Department of Statistics\\
367 Evans Hall \#\  3860\\
U.C. Berkeley CA 94720\\
aldous@stat.berkeley.edu\\bkolesnik@berkeley.edu}

\begin{document}
\maketitle

{\em In memory of Larry Shepp, 1936 - 2013}

\begin{abstract}
On being told that a piece of work he thought was his discovery had duplicated an earlier mathematician's work, Larry Shepp once replied 
``Yes, but when {\em I} discovered it, it {\em stayed} discovered".
In this spirit we give discussion and probabilistic proofs of two related known results 
(Moon 1963, Joe 1988) on random tournaments which seem surprisingly unknown to modern probabilists. 
In particular our proof of Moon's theorem on mean score sequences seems more constructive than previous proofs.
This provides a comparatively concrete introduction to a longstanding mystery, 
the lack of a canonical construction for a joint distribution in the representation theorem for convex order.

\end{abstract}

\section{Introduction}
\subsection{An analogy}
\label{sec:analogy}
As an analogy for the two theorems to be discussed, we recall some very well known facts.  
First note\\
(a) A probability measure $\mu = \mathrm{dist}(X_1,\ldots,X_n)$ on $\Reals^n$ with each $\Ex X_i = 0$ and $\Ex X_i^2 < \infty$ has a covariance matrix 
$\Gamma_{ij} = \Ex X_iX_j <  \infty$. \\
(b) If $\Gamma$ is a symmetric $n \times n$ matrix such that
$f(\bx) = f(x_1,\ldots,x_n) \propto \exp (- \bx^T \Gamma^{-1} \bx/2)$ 
is a well-defined probability distribution, then its covariance matrix is $\Gamma$.\\
What is true, but not obvious {\em a priori}, is that 
essentially\footnote{We need {\em essentially} to handle the degenerate case: precisely, the set of matrices defined by (a) is the closure  
of the set defined by (b).}
 {\em every} covariance matrix (as defined in (a)) can be used to define the Gaussian 
distribution in (b).
This is the Fact  1 for our analogy.  
The reader will likely perceive this fact as a consequence of the multivariate CLT, although there is an arguably conceptually simpler proof (see section~\ref{sec:M-E}). 
Fact 2 is the familiar characterization of covariance matrices as
symmetric positive semi-definite matrices.

Our Theorems \ref{T1} and \ref{T2} may be regarded as analogs of Facts 1 and 2 in a different context.

\subsection{Statements of the two theorems}
In everyday language {\em tournament} usually means a single-elimination tournament.  
In graph theory an $n$-team  {\em tournament} means the set of win-lose results of league play in which each pair of teams plays once.
So there are exactly $2^{n \choose 2}$ tournaments.
Now consider a completely arbitrary probability distribution $\mu$ on the set of $2^{n \choose 2}$ tournaments.
For $1 \le i \le n$
write $x_i$ for the expectation of the number of wins by team $i$.
Call such a sequence $\bx = (x_i, 1 \le i \le n)$ a {\em mean score sequence}.
In this context, an analog (as mathematically tractable, for instance)  of the Gaussian distribution is the following
{\em Bradley--Terry} model.
Take real parameters $(\lambda_i, 1 \le i \le n)$ and let match results be independent with
\begin{equation}
p_{ij} :=\Pr(i \mbox{ beats } j ) = L(\lambda_i - \lambda_j)
\label{BT-def}
\end{equation}
for the logistic function
\[ L(u) := \frac{e^u}{1+ e^u}, \ -\infty < u < \infty . \]
The set $\XX_n^{B-T}$ of mean score sequences arising from the Bradley--Terry model must be a subset of the set $\XX_n$ of all 
mean score sequences.
But it turns out that, as suggested by the Gaussian analogy, we get essentially {\em all} mean score sequences from a Bradley--Terry model.

\begin{Theorem}[Joe \cite{joe88}]
\label{T1}
The closure in $\Reals^n$ of the set $\XX_n^{B-T}$ is $\XX_n$.
\end{Theorem}
We need {\em closure} to get the extreme cases of the mean score sequence 
 where for some $k$,  $x_i=i-1$
for all $i\ge k$. 

In fact Theorem 1 is not explicitly stated and proved in \cite{joe88}, but it 
is a straightforward consequence of results there.
We give both a simple heuristic argument and a careful proof in section \ref{sec:BT}.

As background to the second theorem, recall that the following definition and representation 
(see e.g.\ \cite{MR1278322} section 2.A) 
of {\em convex order} $\preceq$ are useful 
in several areas of probability.

\noindent
{\bf Definition.} For finite mean probability measures on $\Reals$, 
$\mu \preceq \nu$ means 
\begin{equation}
\int \phi d \mu \le   \int \phi d \nu \ \mbox{ for all convex } \phi  \mbox{ such that the integrals exist.} 
 \label{phiY}
 \end{equation}
 
 \noindent
 {\bf Representation (Strassen~\cite{strassen65}). } $\mu \preceq \nu$ if and only if
there exists a joint distribution  $(X,Y)$ with marginal distributions $\mu$ and $\nu$
such that
\begin{equation}
 X = \Ex(Y | X) .
 \label{YZ}
 \end{equation}
 In the special case where $\mu$ and $\nu$ are uniform on $n$-element multisets $\bx$ and $\by$, written by convention in increasing\footnote{{\em Increasing} means non-decreasing.} order
 as $\bx = (x_1, \ldots,x_n)$ and $\by = (y_1,\ldots, y_n)$, the convex order  $\mu \preceq \nu$ is equivalent to the notion 
 that {\em $\bx$ is majorized by $\by$},  also written as $ \bx \preceq \by$, and usually defined by
 \begin{equation}
 \sum_{i=1}^k x_i \ge  \sum_{i=1}^k y_i, \ 1 \le k \le n, \mbox{ with equality for } k = n .
 \label{major}
 \end{equation}  
 The breadth of utility of this notion of majorization, mostly outside probability, is demonstrated in the book \cite{MOA}.
The relatioin $\bx\preceq\by$
can be thought of intuitively as ``$\bx$ is less spread 
out than $\by$'', as discussed in \cite{MOA}.
 
 The second theorem is our loose analog of the characterization of covariance matrices as symmetric positive semi-definite matrices.
 \begin{Theorem}[Moon \cite{moon_thm}]
\label{T2}
 An increasing sequence of real numbers
$\bx = (x_1, x_2, \ldots, x_n)$ is a mean score sequence of some random tournament  if and only if 
$\bx \preceq (0,1,\ldots,n-1)$.
\end{Theorem}

We give proofs of these theorems in sections \ref{sec:proofs} and \ref{sec:BT}, and then discuss previous proofs and the broader context in section \ref{sec:discuss}. 
From that broader context, we feel that the most interesting part of these results 
 is how should one  prove the ``if" part of Theorem~\ref{T2}, so we start by giving two proofs of that.

\section{Proofs of Theorem \ref{T2}}
\label{sec:proofs}

\subsection{First proof of Theorem~\ref{T2}  (``if" part).} 
\label{sec:football}
Given $\bx \preceq (0,1,\ldots,n-1)$, the representation (\ref{YZ}) 
holds with $X$ uniform on $\{x_1, \ldots,x_n\}$ and $Y$ uniform on $(0,1,\ldots,n-1)$, and then 
$\mu_i(\cdot) := \Pr(\cdot | X=x_i)$ define
probability distributions  $\mu_i(\cdot), \ 1 \le i \le n$, on integers 
 $\{0,1,\ldots, n-1\}$ such that\\
 (i) $\mu_i(\cdot)$ has mean $x_i$\\
 (ii) $\sum_i \mu_i(j) = 1$, all $0 \le j \le n-1$.\\

\noindent
The proof is easily understood in terms of a model for football\footnote{Soccer (U.S.)} -- 
note this is separate from the initial ``tournament" model.
\begin{quote}
When teams $i$ and $j$ play, their goal scores are independent with distributions
 $\mu_i$ and $\mu_j$. Define 
 $p_{ij}$ as the mean number of points earned by $i$, when a win earns $1$ point and a tie earns $1/2$ point.
\end{quote}
To complete the proof we need only check\\
(iii) $\sum_{j \ne i} p_{ij} = x_i, \ 1 \le i \le n$, \\
in other words check 
\begin{equation}
\sum_{j \ne i} \chi(\mu_i,\mu_j) = x_i, \ 1 \le i \le n
\label{check}
\end{equation}
where for independent random variables $X$ and $\hat{X}$ with distributions $\nu$ and $\hat{\nu}$ 
we define 
\[ \chi(\nu, \hat{\nu}) = \Pr( X > \hat{X}) +  \sfrac{1}{2} \Pr( X = \hat{X}) . \]
Note that $\chi$ is linear in each argument and that, for the uniform distribution $\lambda$ on  $\{0,1,\ldots, n-1\}$ 
we have
$\chi(\delta_k, \lambda) = (k +  \sfrac{1}{2})/n$.
So by linearity and (i), $\chi(\mu_i, \lambda) = (x_i +  \sfrac{1}{2})/n$.
By (ii) we may write $\lambda = \frac{1}{n} (\mu_i + \sum_{j \ne i} \mu_j)$ and then by linearity
\[ x_i +  \sfrac{1}{2} = \chi(\mu_i,\mu_i) + \sum_{j \ne i} \chi(\mu_i,\mu_j) . \]
But $ \chi(\mu_i,\mu_i) =  \sfrac{1}{2}$ by symmetry, verifying (\ref{check}).

\subsection{Second proof of Theorem~\ref{T2}  (``if" part).} 
A permutation $\pi: \{1,2,\ldots,n\} \to \{1,2,\ldots,n\}$  
defines a special kind of ``season results", in which team $i$ loses to team $j$ if and only if $\pi(i) < \pi(j)$, 
and therefore team $i$ wins exactly  $\pi(i) - 1$ games.
By considering this special type of ``totally ordered" season result, the converse will follow from the lemma below,
because the mean score sequence for the corresponding ``random total order" is 
$(q(i) - 1, 1 \le i \le n)$.

\begin{Lemma}
Let $q:  \{1,2,\ldots,n\} \to [1,n]$ be a function such that 
$q(1+U_n) \preceq 1+U_n$, where $U_n$ has uniform distribution on $\{0,1,\ldots,n-1\}$. 
Then there exists a probability distribution over permutations $\pi$ such that 
\[ \Ex \pi(i) = q(i), 1 \le i \le n . \]
\end{Lemma}

\begin{proof}
\noindent Write $I_n$ and $J_n$ for random variables with the uniform distribution on $\{1,2,\ldots,n\}$. 
The relation $q(I_n) \preceq J_n$ is equivalent, using the representation (\ref{YZ}), to saying that we can construct a joint 
distribution for $(I_n, J_n)$  such that 
\[ \Ex (J_n | I_n = i) = q(i), \ 1 \le i \le n . \]
Now the matrix with entries
\[ p_{ij}:= n \Pr(I_n = i, J_n = j) \]
is doubly stochastic, so by Birkhoff's theorem it is a mixture of permutation matrices.  
In other words, there is a probability distribution over permutations $\pi$ such that 
$p_{ij} = \Pr(\pi(i) = j)$.
But this just says 
\[ \Pr(\pi(i) = j) = \Pr(J_n = j | I_n = i) \]
and so 
\[ \Ex \pi(i) = \Ex (J_n | I_n = i) = q(i).  \qedhere \]
\end{proof}

\subsection{Proof of Theorem~\ref{T2}  (``only if" part).} 
First consider the deterministic case, so there is an integer-valued sequence $0 \le x^*_1\le x^*_2 \le \cdots \le x^*_n$ of wins.
Fix a convex function $\phi$. 

Suppose $x^*_1  \ge 1$.  Change the results of the games won by team $1$, one game at a time, to make them a loss for team $1$.
At each such step $x^*_1$ decreases by $1$ and some $x^*_i$ increases by $1$, and by convexity the value of 
$\sum_i  \phi(x^*_i)$ can only increase.  
Continue until reaching a configuration with $x^*_1 = 0$.  Now suppose $x^*_2 \ge 2$. 
Again change the results of the games won by team $2$ against teams $i > 2$, one game at a time, to make them a loss for team $2$. 
Again this can only increase $\sum_i  \phi(x^*_i)$.  
Continue until reaching a configuration with $x^*_1  = 0$ and $x^*_2 = 1$.
Eventually we reach the configuration with 
$x^*_i = i-1, 1 \le i \le n$.  So the original configuration satisfies $\sum_i  \phi(x^*_i) \le \sum_i  \phi(i-1)$, establishing (\ref{phiY}) in the 
deterministic case.

In the random case write $X_i$ for the random number of wins by team $i$.  So $x_i = \Ex X_i$, and by Jensen's inequality
$\phi(x_i) \le \Ex \phi(X_i)$.  So 
\[  \sum_i \phi(x_i) \le \sum_i \Ex \phi(X_i) = \Ex \left[ \sum_i \phi(X_i) \right] \] 
and the quantity in brackets is bounded by $ \sum_i  \phi(i-1)$, by the deterministic case.
This establishes (\ref{phiY}) in general.

\section{The Bradley--Terry model}
\label{sec:BT}


\subsection{The max-entropy heuristic}
There is a one sentence explanation of Theorem \ref{T1}. 
\begin{quote}
If $\bx$ is a mean score sequence then by definition there exists a matrix with non-negative off-diagonal entries $p_{ij} = 1 - p_{ji}$ satisfying
the constraints
$\sum_{j \ne i} p_{ij} = x_i, 1 \le i \le n$; and the max-entropy  such matrix is of the Bradley--Terry form (\ref{BT-def}).
\end{quote}
To elaborate, the problem
\begin{quote}
maximize $- \sum_{i\ne j} p_{ij} \log p_{ij} $ subject to $\sum_{j \ne i} p_{ij} = x_i, 1 \le i \le n$
\end{quote}
is solved, in classical applied mathematics, by introducing Lagrange multipliers $\lambda_i, 1 \le i \le n$ 
and solving
{\small 
\begin{quote}
maximize $- \sum_{i\ne j} p_{ij} \log p_{ij} + \sum_i \lambda_i ( \sum_{j \ne i} p_{ij} - x_i)$ over $\lambda_i, 1 \le i \le n$.
\end{quote}
}
\noindent
By setting $\frac{d}{dp_{ij}} ( \cdot) = 0$ the solution satisfies
\[ - \log p_{ij} +  \log (1- p_{ij}) + \lambda_i - \lambda_j = 0 \]
implying that the matrix $(p_{ij})$ is indeed of the Bradley--Terry form (\ref{BT-def}).

We learned this max-entropy argument from Joe \cite{joe88}.
How much detail needs to be added to make a completely rigorous proof is a matter of taste; 
we give a rather fussy argument next.

\subsection{Strong stochastic transitivity}
In  statistical modeling contexts such as \cite{aditya} one  says that $p_{ij}$ on tournaments of size $n$ has the 
{\em strong stochastic transitivity}
(SST) property if (after relabelling $[n]$ if necessary)   
\begin{equation}
\mbox{$p_{ij}$ is increasing 
in $i$, for any fixed $j$.}
\label{SST-1}
\end{equation} 
\begin{Lemma}[{Joe~\cite{joe88} Theorems~2.3 and 2.7}]\label{L_joe}
Let $\bx \preceq (0,1,\ldots,n-1)$. 
Suppose that 
$\psi$ is strictly convex. If $p_{ij}^*$ minimizes $\sum_{i\neq j}\psi(p_{ij})$
over $p_{ij}$ on tournaments of size $n$ 
with mean score sequences $\bx$, then 
$p_{ij}^*$ has SST. 
\end{Lemma}
We note that in \cite{joe88} and earlier combinatorial literature the SST property is instead
defined to mean 
\begin{equation}\label{E_SST}
\min\{p_{ij},p_{jk}\}\ge1/2\Rightarrow
p_{ik}\ge \max\{p_{ij},p_{jk}\}, 
\end{equation}
It is apparently well known that the two definitions are equivalent but we cannot find a published proof
(it does not appear in the old survey  \cite{fishburn}) so we have included a proof in the Appendix. 




\begin{proof}[Proof of Theorem~\ref{T1}]
We show that if ${\bf x}\prec (0,1,\ldots,n-1)$, in the sense that 
$\sum_{i=1}^k x_i > {k\choose 2}$ for $1\le k<n$ and 
$\sum_{i=1}^n x_i = {n\choose 2}$, then $p_{ij}=L(\lambda_i-\lambda_j)$
for some $(\lambda_i,1\le i\le n)$. The theorem follows. 

To this end, consider 
minimizing $\sum_{i\neq j} p_{ij}\log p_{ij}$, subject
to all $p_{ij}\le 1$, 
$p_{ij}+p_{ji}=1$  
 and $x_i=\sum_{j\neq i}p_{ij}$.  
Since ${\bf x}\preceq(0,1,\ldots,n-1)$, 
a solution $p^*_{ij}$ exists. 
We claim that all $p^*_{ij}\in(0,1)$. This concludes the proof, noting that  
the stationary point of
\[
\sum_{i\neq j} p_{ij}\log p_{ij}
+\sum_i \lambda_i(x_i-\sum_{j\neq i}p_{ij})
\]
corresponds to some $p^*_{ij}=L(\lambda^*_i-\lambda^*_j)$ 
with all  
$x_i=\sum_{j\neq i}L(\lambda^*_i-\lambda^*_j)$.  

To establish the claim, suppose  
towards a contradiction that some $p^*_{ij}=1$. 
Since $\psi(x)=x\log x$ is strictly convex, 
$p^*_{ij}$ has SST
by  
Lemma~\ref{L_joe}. 
Let $u$ be the minimal index such that 
$p^*_{uj}=1$ for some $j$. Let 
$v$ the maximal index such that 
$p^*_{uv}=1$. Note that since $x\prec(0,1,\ldots n-1)$  
we have $v<u-1$, as else, since $p^*_{ij}$ has SST,
we would find that 
$p^*_{ij}=0$ for all $i\le u-1$ and $j\ge u$, 
and so $\sum_{i=1}^{u-1}x_i={u-1\choose 2}$.

Next, 
consider $q_{ij}$ obtained from $p^*_{ij}$ 
by decreasing $p^*_{uv}=1$ by $\eps$ and 
increasing $p^*_{u,v+1}$ and $p^*_{v+1,v}$
by $\eps$.
That is, let\\
(i) $q_{vu}=1-q_{uv}=\eps$, \\
(ii) $q_{u,v+1}=1-q_{v+1,u}=p^*_{u,v+1}+\eps$,\\
(iii) $q_{v+1,v}=1-q_{v,v+1}=p^*_{v+1,v}+\eps$,\\
and $q_{ij}=p^*_{ij}$ for all other $i,j$. 
Since $v<u-1$, and so $v+1\neq u$,  
all 
\[
\sum_{j\neq i} q_{ij}=\sum_{j\neq i} p^*_{ij}=x_i.
\]
Moreover, by the choice of $u$ and $v$, and since $v+1<u$, note that
$p^*_{u,v+1}<1$ and 
$p^*_{v+1,v}<1$. 
Therefore all $q_{ij}\le1$, for all small $\eps>0$. 
Hence $q_{ij}$ is a distribution with mean score
sequence ${\bf x}$. 

Finally, observe that,  for $\alpha, \beta \in (0,1)$, differentiating 
\[
\ell(\eps)+\ell(\alpha+\eps)+\ell(\beta+\eps)
\]
with respect to $\eps$, 
where $\ell(x)=x\log{x}+(1-x)\log(1-x)$, 
we obtain 
\[
\log\left(
\frac{\eps}{1-\eps}
\frac{\alpha+\eps}{1-\alpha-\eps}
\frac{\beta+\eps}{1-\beta-\eps}
\right) <0
\]
for all small $\eps>0$. 
It follows that 
\[
\sum_{i\neq j} q_{ij}\log q_{ij}
<
\sum_{i\neq j} p^*_{ij}\log p^*_{ij}
\]
for all small $\eps>0$, 
contradicting the minimality of $p^*_{ij}$. 
\end{proof}

\subsection{Comments on max-entropy}
\label{sec:M-E}
The section~\ref{sec:analogy} analogy, that  (a) is the closure of the set defined by (b), can also be proved by max-entropy instead of the multivariate CLT.
And in the context of this paper, 
we could call $p_{ij}$ {\em Cauchy} if $p_{ij}=C(\alpha_i-\alpha_j)$,
where $C(x)=2^{-1}+\pi^{-1}\arctan{x}$ is the CDF of a standard Cauchy. 
The same argument  above works, replacing the entropy function $u\log u$ by the function  $-(2\pi)^{-1}\log\sin(u\pi)$, to show that the set of mean score sequences from Cauchy matrices is dense in the set of all mean score sequences.


\section{Discussion}
\label{sec:discuss}
One purpose of this paper is simply to juxtapose Theorems \ref{T1} and \ref{T2}, which come from rather different research communities.
As noted below, Moon's theorem has many proofs using different textbook theorems -- 
and indeed could be used as a running example in a first undergraduate course in discrete mathematics.
As probabilists we wanted to find ``probabilistic" arguments, and our two ``if" proofs in section \ref{sec:proofs} are (we believe) new and probabilistic.
In particular the ``football story" in section \ref{sec:football} seems memorable. 
The relevant discrete mathematics literature starts with 
Landau's theorem  \cite{landau}, the deterministic analog of Moon's theorem  characterizing score sequences of non-random tournament outcomes.
Moon's original proof \cite{moon_thm}, and  a more general version (their Theorem 3.9) in
Moon and Pullman \cite{pullman},
used network flow feasibility properties, and later
 Bang and Sharp \cite{bang} used  Hall's theorem on systems of distinct representatives, 
Cruse \cite{cruse} used linear programming methods, 
and Thornblad  \cite{thornblad} derived it from the deterministic case \cite{landau} via a rather lengthy argument.  
To us it is more natural to explicitly exploit the Strassen representation, as our two ``if" proofs do in different ways. 
The ``only if'' part could alternatively  be proved as an easy consequence of Landau's theorem and \cite{MOA} Proposition~12.D.1.
Our proof (by ``Robin Hood moves''  \cite{MOA}) 
avoids Landau's theorem.


As mentioned before, convex order and the general Strassen representation are useful tools in several areas of probability theory.
Because an arbitrary distribution on $\Reals$ can be approximated by uniform distributions on $n$-element sets, 
the general case is conceptually very similar to the majorization case $\bx \preceq \by$ defined at (\ref{major}).
This setting is treated at great length in the  book \cite{MOA}.
Instead of the  ``coupling of random variables" picture natural to modern probabilists,
the representation is stated there in the equivalent form  
\begin{equation}
 \mbox{ if }    \bx \preceq \by \mbox{ then }    \bx = A \by   \mbox{ for some doubly stochastic matrix $A$}.
 \label{rep2}
 \end{equation}
 There are constructive proofs of this (e.g. \cite{MOA} Theorem 2.B.2).
 Combined with our ``football" proof we obtain a completely constructive proof of Moon's theorem. 
 To our knowledge, previous proofs cannot be made constructive so easily.

To us,  the most interesting part of the bigger picture surrounding convex order is that there is apparently no ``canonical"  choice of joint distribution in (\ref{YZ}, \ref{rep2}):
proofs may be constructive but they involve rather arbitrary  choices and the resulting joint distributions are not easily described.
Recent literature on {\em peacocks}  \cite{peacock}
studies continuous-parameter processes increasing in convex order, via many different constructions, 
and ideas from that literature might be relevant in our context.
 

We encountered this field while exploring the Bradley--Terry model as a basic mathematical toy model for sports results -- see  \cite{me-elo} and   \cite{kiraly} 
for references to the extensive literature in that field.
Abstractly there is a map $G$ from the set of 
$ - \infty < \lambda_1 \le \lambda_2 \le \cdots \le \lambda_n < \infty$, centered by requiring $\sum_i \lambda_i = 0$, to $\bx \in \Reals^n$
defined by
\[ x_i = \sum_{j \ne i} L(\lambda_i - \lambda_j). \]
This map has some range  $\XX_n^{B-T}$,
but it is hard to see directly from that definition what is the range of $G$.  
Theorem~\ref{T1} answers that question.
However the inverse function $G^{-1}$ giving $(\lambda_1,\ldots,\lambda_n)$ in terms of $\bx \in \XX_n$ 
remains obscure; we do not have an explicit formula.
So we cannot say anything about how the win-probabilities $p_{ij} = L(\lambda_i - \lambda_j)$ depend on $\bx$.
In particular, we do not know whether these  win-probabilities $p_{ij} $ can always be obtained within the ``football model" by some choice of 
distributions $\mu_i$ (equivalently of matrix $A$ at (\ref{rep2})).


\appendix
\section{Appendix: Proof of equivalence of the two definitions of SST}
Suppose that $p_{ij}$ has SST as defined by (\ref{SST-1}).
Then  $p_{ij}\ge1/2$ 
(for which necessarily $i\ge j$)
are increasing in $i$ for any $j$,  
and decreasing in $j$ for any $i$, giving 
\eqref{E_SST}. 

Conversely, suppose we have 
\eqref{E_SST}. 
Consider the graph $G$ with directed edges $j\to i$ for $p_{ij}\ge1/2$. 
By induction, all induced $H\subset G$
have a vertex $i_*$ such that $j\to i_*$ for all $j\neq i_*\in H$.
Indeed, suppose this holds for all such subgraphs of size $m<n$.
Let $H$ be an induced subgraph of size $m+1$, and consider $H'=H-v$
for some $v\in H$. Let $i_*'$ be such that $j\to i_*'$ for all $j\neq i_*'\in H'$. 
If $v\to i_*'$ in $H$ then let $i_*=i_*'$. 
Otherwise, if $v\not\to i_*'$ (and so $i_*'\to v$) in $H$, then 
\eqref{E_SST} implies that 
all other $j\to v$ are in $H$ 
(as else, 
since $j\to i_*'$, the presence of some $v\to j$  in $H$
and  \eqref{E_SST} would imply that 
$v\to i_*'$ in $H$, contrary to our assumption),  
and so let $i_*=v$.
This completes the induction.
The case $H=G$ gives $i_*$ such that $p_{i_*j}\ge1/2$ for all $j$. 
Therefore, we can recursively relabel $[n]$ so that $p_{ij} \ge 1/2$ for all pairs $i>j$.
Then assumption \eqref{E_SST} directly implies the
SST property in the sense of (\ref{SST-1}).

\end{document}